\documentclass[12pt]{article}

\usepackage{a4wide}
\usepackage{array,amsmath,amsfonts,amssymb}
\usepackage{bm}
\usepackage{color}
\usepackage[english]{babel}
\usepackage{graphicx}
\usepackage{todonotes}

\numberwithin{equation}{section}

\newcommand{\bfI}{\boldsymbol{I}}

\newcommand{\tn}{|||}

\newcommand{\mcN}{\mathcal{N}}

\newcommand{\mcT}{\mathcal{T}}

\newcommand{\nablas}{\nabla_\Gamma}

\newcommand{\IR}{\mathbb{R}}
\newtheorem{lem}{Lemma}[section]

\newtheorem{thm}{Theorem}[section]
\newtheorem{rem}{Remark}[section]

\newenvironment{proof}{\noindent \newline {\bf Proof.}}
{\hfill \mbox{\fbox{} } \newline}

\title{\bf A Simple Finite Element Method for Elliptic 
Bulk Problems with Embedded Surfaces\thanks{This research was
  supported in part by the Swedish Foundation for Strategic Research
  Grant No.\ AM13-0029, the Swedish Research Council Grants Nos.\
  2011-4992, 2013-4708, and the Swedish Research Programme
  Essence. The first author was supported in part by the EPSRC grant EP/P01576X/1.}}
\date{\today}
\author{Erik Burman\footnote{Department of Mathematics, University 
 College London, Gower Street, London WC1E 6BT, UK, {\tt e.burman@ucl.ac.uk}.}
 \mbox{ } 
Peter Hansbo\footnote{Department of Mechanical Engineering, J\"onk\"oping University, 
SE-55111 J\"onk\"oping, Sweden, {\tt Peter.Hansbo@ju.se}.}
 \mbox{ }  
Mats G. Larson\footnote{Department of Mathematics and Mathematical
  Statistics, Ume{\aa} University, SE-90187 Ume{\aa}, Sweden, {\tt mats.larson@umu.se}. }
}
\begin{document}

\maketitle

\begin{abstract}
In this paper we develop a simple finite element method for simulation of embedded 
layers of high permeability in a matrix of lower permeability using a basic model of 
Darcy flow in embedded cracks. The cracks are allowed to cut through the mesh 
in arbitrary fashion and we take the 
flow in the crack into account by superposition. The fact that we use continuous elements 
leads to suboptimal convergence due to the loss of regularity across the crack. We 
therefore refine the mesh in the vicinity of the crack in order to recover optimal order 
convergence in terms of the global mesh parameter. The proper degree of refinement is 
determined based on an a priori error estimate and can thus be performed before the 
actual finite element computation is started. Numerical examples showing this effect 
and confirming the theoretical results are provided. The approach is easy to implement 
and beneficial for rapid assessment of the effect of crack orientation and may for example 
be used in an optimization loop.
\end{abstract}

\section{Introduction}

\paragraph{New Contributions.}
In this contribution we consider a basic elliptic problem with an embedded 
interface with high permeability, which may be used to model the pressure in 
a medium with cracks or the temperature in composite materials. Our approach 
is to use a continuous piecewise linear finite element space and simply insert 
this space into the weak formulation of the continuous problem which consists 
of a sum of a form on the bulk domain and a form on the interface. Note that 
the interface cuts through the mesh in an arbitrary way but we avoid using 
computations on cut elements and instead compensate the lack of regularity 
across the interface using a mesh which is adapted close to the interface. This 
approach leads to a scheme which is very easy to implement.

We derive a priori error estimates which shows that the meshsize for elements 
close to the interface $h_\Gamma \sim h^2$ where $h$ is the global mesh parameter 
used in the bulk mesh. Such a pre-refinement of the mesh leads to optimal order 
a priori error estimates in terms of the global mesh parameter. Note that no adaptive 
algorithm is used instead we just split elements that intersect the interface until they 
are small enough. We start with a quasi uniform mesh and refine to obtain a 
conforming locally quasi uniform mesh for instance using an edge bisection algorithm. 
   
In forthcoming work we consider schemes using cut elements which does not 
require adaptive mesh refinement  and also works for higher order elements. The 
method proposed here is however attractive due to its simplicity and may be an 
interesting alternative in situations where one does need very accurate solutions 
for instance in the presence of uncertainties or very complicated networks of 
interfaces, or for optimization purposes.

\paragraph{Earlier Work.} The model we use is essentially the one proposed by 
Capatina et al. \cite{CaLuElBa16}. More sophisticated models have been proposed, e.g., 
in \cite{AnBoHu09,FoFuScRu14,FrRoJeSa08,MaJaRo05}, in particular allowing for jumps in the solution
across the interfaces. To allow for such jumps, one can either align the mesh with the interfaces, as in, e.g., \cite{HaAsDaEiHe09}, or use extended finite element techniques, cf. 
\cite{BurClaHan15,CaLuElBa16,DASc12,DeFuSc17}. Our approach, using a continuous approximation, does not allow for
jumps, but we shall return to this question in a companion paper. 

The approach of superimposing lower dimensional structures 
independently of the mesh was recently introduced in the context of structural mechanics in \cite{BuHaLa17,CeHaLa16}.

\paragraph{Outline.} In Section 2 we formulate the model problem, its weak form, 
and investigate the regularity properties of the solution, in Section 3 we formulate 
the finite element method, in section 4 we derive error estimates, and in Section 5 
we present numerical examples including a study of the convergence and a more 
applied example with a network of cracks.

\section{Model Problem}

\paragraph{Strong Formulation.}
Let $\Omega$ be a convex polygonal domain in $\IR^d$, with  
$d=2$ or $3$. Let $\Gamma$ be a smooth embedded interface
in the interior of $\Omega$ without boundary. Then  $\Gamma$ 
partitions $\Omega$ into two subdomains $\Omega_1$ and 
$\Omega_2$, where $\Omega_2$ is the domain enclosed by 
$\Gamma$. Let $n_i$ be the exterior unit normal to $\Omega_i$. 
See Figure \ref{fig:domains}.

Consider the problem: find $u:\Omega \rightarrow \IR$ such that 
\begin{alignat}{3}\label{eq:strong-problem-a}
-\nabla \cdot a \nabla u &= f,\qquad &\text{in $\Omega$}
\\ \label{eq:strong-problem-b}
-\nabla_\Gamma \cdot a_\Gamma  \nabla_\Gamma u &= f_\Gamma 
+ [n \cdot a \nabla u] ,\qquad &\text{on $\Gamma$}
\\ \label{eq:strong-problem-c}
[u] &= 0,\qquad &\text{on $\Gamma$}
\\ \label{eq:strong-problem-d}
u &=0, \qquad  &\text{on $\partial \Omega$}
\end{alignat}
where $a|_{\Omega_i} = a_i$ are given constants, and 
$f\in L^2(\Omega)$,  $f_\Gamma \in L^2(\Gamma)$ are given functions. 
We also used the notation  $\nabla_\Gamma = P \nabla$ for the tangential 
gradient where $P= I - n \otimes n$ is the tangent projection. The jump in 
the primal variable  across the interface is defined for $x \in
\Gamma$ by $[v] :=
\lim_{\epsilon \rightarrow 0^+} (v(x +\epsilon n_1) - v(x +\epsilon n_2))$ and that of 
the normal flux is defined by
$[n\cdot a \nabla v]= n_1 \cdot a_1 \nabla v_1 + n_2 \cdot a_2 \nabla v_2$, 
where we recall that $n_2 = - n_1$ on $\Gamma$.
%
\paragraph{Function Spaces.} We adopt the usual notation $H^s(\omega)$ for the 
Sobolev space of order $s$ on the set $\omega$ and we have the special spaces 
$H^1_0(\omega) = \{ v \in H^1 (\omega) : \text{$v = 0$ on $\partial \omega$} \}$ 
and $L^2(\omega) = H^0(\omega)$. For a normed vector space $V$ we let 
$\| \cdot \|_V$ denote the norm on $V$ and we use the simplified notation 
$\|v\|_{L^2(\omega)} = \| v \|_\omega$. We denote the $L^2$-scalar
product over $\omega \subset \mathbb{R}^d$ or $\omega \subset
\mathbb{R}^{d-1}$ by $(\cdot,\cdot)_\omega$.

\paragraph{Weak Formulation.}
Multiplying (\ref{eq:strong-problem-a}) by $v\in V = H^1_0(\Omega)\cap H^1(\Gamma)$ 
and using Green's formula we obtain the weak form
\begin{align}\label{eq:weak-start}
(f,v)_\Omega 
&= \sum_{i=1}^2 (-\nabla \cdot a_i \nabla  u_i, v_i)_{\Omega_i}
\\
&= \sum_{i=1}^2 (a_i \nabla  u_i, \nabla v_i)_{\Omega_i} - (n_i \cdot a \nabla u_i, v_i)_\Gamma 
\\
&=(a \nabla u,\nabla v)_\Omega 
      - ([n \cdot a \nabla u ], v )_\Gamma
\\
&= (a \nabla u,\nabla v)_\Omega 
- (f_\Gamma + \nabla_\Gamma \cdot a_\Gamma \nabla_\Gamma u, v )_\Gamma
\\ \label{eq:weak-end}
&= (a \nabla u,\nabla v)_\Omega 
+ ( a_\Gamma \nabla_\Gamma u, \nabla_\Gamma v )_\Gamma 
- (f_\Gamma,v)_\Gamma
\end{align}
where we used the fact that the boundary contributions on $\partial \Omega$ vanish due 
to the boundary condition and then we used (\ref{eq:strong-problem-b}). We thus 
arrive at the weak formulation: find $u \in V$ such that 
\begin{equation}\label{eq:weak}
A(u,v) = L(v) \qquad v \in V
\end{equation}
where 
\begin{align}
A(u,v)&=  (a \nabla u,\nabla v)_\Omega 
+ ( a_\Gamma \nabla_\Gamma u, \nabla_\Gamma v )_\Gamma
 \\
L(v) &= (f,v)_\Omega +  (f_\Gamma,v)_\Gamma
\end{align}
Introducing the energy norm 
\begin{equation}
\tn v \tn^2  = A(v,v)
\end{equation}
on $V$, it follows using the Poincar\'e inequality $\| v \|_\Omega \lesssim \| \nabla v \|_{\Omega}$, 
which holds since $v=0$ on $\partial \Omega$, and the trace inequality 
$\| v \|_\Gamma \lesssim \|v\|_{H^1(\Omega_2)}$, that 
\begin{equation}
\tn v \tn^2 \sim \| v \|^2_{H^1(\Omega)} + \| v \|^2_{H^1(\Gamma)}
\end{equation}
and hence $\tn \cdot \tn$ is a norm on $V$. The form $A$ is a scalar product 
on $V$ and by definition $A$ is coercive and continuous on $\tn \cdot \tn$. Therefore it follows from the 
Lax-Milgram Lemma that there is a unique solution $u \in V$ to (\ref{eq:weak}).

\paragraph{Regularity Properties.}
We have the elliptic regularity estimate 
\begin{equation}\label{eq:ell-reg}
\| u \|_{H^2(\Omega_1)} + \| u \|_{H^2(\Omega_2)}  + \| u \|_{H^2(\Gamma)} 
\lesssim    
\| f \|_\Omega + \| f_\Gamma \|_\Gamma
\end{equation}
To verify (\ref{eq:ell-reg}) we let $u_i \in H^1_0(\Omega_i)$ solve 
\begin{equation}
(a_i \nabla u_i, \nabla v)_{\Omega_i} = (f,v)_{\Omega_i}\qquad \forall v \in H^1_0(\Omega_i)
\end{equation}
Then we have 
\begin{equation}\label{eq:ell-reg-sub}
\| u_i \|_{H^2(\Omega_i)} \lesssim \| f \|_{\Omega_i} \qquad i = 1,2
\end{equation}
Observe that by the boundary conditions and the regularity of $u_i$ we
have that $\nablas u_i = 0$, $i=1,2$.
Next writing $u = u_\Gamma + u_1 + u_2$ we find using the equation that 
$u_\Gamma \in V$ satisfies
\begin{align}\nonumber
-\nablas \cdot a_\Gamma \nablas u _\Gamma = -\nablas \cdot a_\Gamma \nablas u 
&=  f_\Gamma + [ n\cdot a\nabla u ]
\\ \label{eq:u-Gamma}
&= f_\Gamma + n_1 \cdot (a_1\nabla  u_1 - a_2\nabla u_2) 
+ [n \cdot a \nabla u_\Gamma]  \qquad \text{on $\Gamma$}
\end{align}
and 
\begin{equation}
-\nabla \cdot a_i \nabla u_\Gamma = 0 \qquad \text{on $\Omega_i$, $i=1,2$}
\end{equation}
Using (\ref{eq:ell-reg-sub}) we conclude that 
\begin{equation}
n_1 \cdot (a_1\nabla  u_1 - a_2\nabla u_2) |_\Gamma \in H^{1/2}(\Gamma)
\end{equation}
Furthermore, using that $u_\Gamma \in H^1(\Gamma)$, since $u_\Gamma \in V$,  it follows that  
$u_\Gamma |_{\Omega_i}  \in H^{3/2}(\Omega_i)$, $i=1,2,$ and therefore 
\begin{equation}
[n \cdot a \nabla u_\Gamma] \in H^{1/2}(\Gamma)
\end{equation}
Thus using elliptic regularity we find that 
\begin{equation}
u_\Gamma|_\Gamma \in H^2(\Gamma)
\end{equation}
since the right hand side of (\ref{eq:u-Gamma}) is in $L^2(\Gamma)$. 
Collecting the bounds we obtain the regularity estimate 
\begin{equation}
\| u_\Gamma \|_{H^2(\Gamma)} 
+ 
\sum_{i=1}^2 \| u_\Gamma \|_{H^{5/2}(\Omega_i)} + \| u_i \|_{H^2(\Omega_i)} 
\lesssim 
\| f \|_\Omega + \| f_\Gamma \|_\Gamma
\end{equation} 
where we note that we have stronger control of $u_\Gamma$ on the 
subdomains.

\begin{rem} Note that if we instead take $f \in H^{-1/2}(\Gamma)$ we will have 
$u_\Gamma|_\Gamma \in H^{3/2}(\Gamma)$ and $u_\Gamma  \in H^2(\Omega_i)$ 
and the estimate 
\begin{equation}\label{eq:ell-reg-alt}
\| u \|_{H^2(\Omega_1)} + \| u \|_{H^2(\Omega_2)}  + \| u \|_{H^{3/2}(\Gamma)} 
\lesssim    
\| f \|_\Omega + \| f_\Gamma \|_{H^{-1/2}(\Gamma)}
\end{equation}
\end{rem}

\section{The Finite Element Method} 
To design a finite element method for the problem we use the classical
approach restricting the weak formulation \eqref{eq:weak} to a suitably chosen
finite dimensional subspace of $V$.
To this end let
\begin{itemize}
\item  $\mcT_h$ be a locally quasi uniform conforming mesh on $\Omega$, 
consisting of shape regular simplices with element size $h_T$ and let 
$h = \max_{T \in \mcT_h} h_T$ be the global mesh parameter.
\item  $V_h$ be a finite element space consisting of continuous piecewise 
linear polynomials on $\mcT_h$.
\item $\mcT_h(\Gamma)$ denote the set of elements intersected by the interface:
\[
\mcT_h(\Gamma):= \{T \in \mcT_h : K \cap \Gamma \ne \emptyset\}
\]
\end{itemize}
The finite element method takes the form: find $u_h \in V_h$ such that 
\begin{equation}\label{eq:fem}
A(u_h,v) = L(v) \qquad v \in V_h
\end{equation}

\section{Error Estimates}

\subsection{Preliminaries}
\begin{itemize}
\item Let $\rho$ be the signed distance function associated with 
$\Gamma$, negative in $\Omega_1$ and positive in $\Omega_2$. We 
then have $n = \nabla \rho$ where $n = n_1$ is the unit normal direction 
exterior to $\Omega_1$.
\item For $\zeta >0$ let define a tubular neighborhood around $\Gamma$ by
\[
U_\zeta(\Gamma) := \{x \in \Omega: \min_{x_\Gamma \in \Gamma}
\|x - x_\Gamma\|_{\IR^d}\leq \zeta\}.
\]
\item There is $\delta_0>0$ such that for each $x \in U_{\delta_0}(\Gamma)$ 
there is a unique point $p(x)\in \Gamma$ such that $\|x - p(x)\|_{\IR^d}$ is 
minimal called the closest point. We also have the formula
\begin{equation}\label{eq:closest-point}
p(x) = x - \rho(x) n(p(x)) 
\end{equation}
for the so called closest point mapping $p: U_{\delta_0}(\Gamma) \rightarrow \Gamma$ 
\item Let $v^e = v \circ p$ be the extension of $v$ from $\Gamma$ to 
$U_{\delta_0}(\Gamma)$. We then have 
\begin{equation}\label{eq:ext-stab}
\| v^e \|_{U_{\delta}(\Gamma) } \lesssim \delta^{1/2} \| v \|_{\Gamma}
\end{equation}
\item The tangential gradient is defined by
\begin{equation}
\nablas v = P \nabla v^e
\end{equation}
where we recall that $P(x) =I - n(x) \otimes n(x)$ is the projection onto the tangent 
plane $T_x(\Gamma)$ to $\Gamma$ at $x$.
\end{itemize}

\subsection{Interpolation}
We introduce the following concepts.
\begin{itemize}
\item 
Let $\pi_h:L^2(\Omega) \rightarrow V_h$ be the Cl\'ement interpolant which 
satisfies the interpolation error estimate
\begin{equation}\label{eq:interpol-sz}
\| u - \pi_h u \|_{H^s(T)} \lesssim h^{t - s}\| u \|_{H^{t}(\mcN_h(T))} \qquad 0\leq s \leq t \leq 2
\end{equation}
where $\mcN_h(T) \subset \mcT_h$ is the set of all elements which are node neighbors 
of $T$.

\item In order to account for the fact that the exact solution $u$ is not in regular across 
the interface we construct an interpolation operator which is modified close to the interface. 
Essentially we interpolate on an extension of $u|_\Gamma$ in the neighborhood of $\Gamma$ 
and on $u$ outside of $\Gamma$.
Let $\chi:[0,1] \rightarrow [0,1]$ be a smooth function such that $\chi = 0$ on $[2/3,1]$, 
and $\chi=1$ on $[0,1/3]$. On $U_\delta(\Gamma)$ let $\chi_\delta(x) = \chi(|\rho(x)|/\delta)$ 
and on $\Omega \setminus U_\delta(\Gamma)$  let $\chi_\delta(x)= 0$. Define the 
interpolant 
\begin{equation}\label{eq:interpol-modified}
I_h v = 
\pi_h (v(1-\chi_{\delta})  + v^e\chi_{\delta})
=
\pi_h (v + (v^e - v)\chi_{\delta})  
\end{equation}
Note that with this construction we essentially interpolate $u^e$ close to $\Gamma$ 
and $u$ outside of $\Gamma$.

\item 
We consider meshes that are refined in the vicinity of the interface. More precisely, we 
assume  that there are two mesh parameters $h_\Gamma$ and $h$ such that 
\begin{equation}
\begin{cases}
h_T \lesssim h_\Gamma & T \in \mcN_h(\mcT_h(\Gamma))
\\
h_T \lesssim h & T \in \mcT_h \setminus  \mcN_h(\mcT_h(\Gamma))
\end{cases}
\end{equation}

\item We chose $\delta$ in the definition (\ref{eq:interpol-modified}) of $I_h$ in such 
a way that 
\begin{equation}\label{eq:chideltaequalsone}
\mcN_h(\mcT_h(\Gamma))  \subset U_{\delta/3}(\Gamma)
\end{equation}
which means that $\chi_\delta =1$ on $\mcN_h(\mcT_h(\Gamma))$.  
We note that (\ref{eq:chideltaequalsone}) then implies that we may take 
$\delta \sim h_\Gamma$ in the definition of $\chi_\delta$.
  
\end{itemize}  

\begin{rem} We note that the total number of degrees of freedom $N$ is 
related to the global mesh parameter as follows
\begin{equation}
N \sim h^{-d} + h_\Gamma^{-(d-1)} \sim h^{-d} + h^{-2(d-1)}
\end{equation}
Thus we find that for $d=2$ we have $N\sim h^{-2}$, which is equivalent to the unrefined mesh, 
and for $d=3$ we have $N\sim h^{-4}$, which is slightly more expensive compared to the unrefined 
mesh which scales as $h^{-3}$.  
\end{rem} 
  
\begin{lem} There is a constant such that  
\begin{align}\label{eq:interpol-energy-cont}
\tn v - I_h v \tn
\lesssim  
(h + h_\Gamma^{1/2} )(\| v \|_{H^2(\Omega_1)} 
+ \| v \|_{H^2(\Omega_2)} ) + h_\Gamma \| v \|_{H^2(\Gamma)}
\end{align} 
\end{lem}

\begin{proof}
Using the definition of $I_h$ we have 
\begin{align}
v - I_h v 
&=
(v^e - v ) \chi_{\delta} 
+ (I-\pi_h) ( v + (v^e - v)\chi_{\delta}) 
\end{align} 
and thus  
  \begin{align}
 \| \nabla (v  - I_h v) \|_\Omega + \| \nablas (v - I_h v) \|_{\Gamma} 
 &= 
 \| \nabla ( (v^e - v) \chi_{\delta} ) \|_\Omega 
\\ \nonumber
&\qquad + 
 \| \nabla (I-\pi_h) ( v + (v^e - v)\chi_{\delta} ) \|_\Omega
\\ \nonumber
&\qquad +
 \| \nablas (I-\pi_h) ( v + (v^e - v)\chi_{\delta} ) \|_\Gamma
\\
&=I + II + III
 \end{align}

\paragraph{Term $\bfI$.} Using the product rule and the triangle inequality 
\begin{align}
 \| \nabla ((v^e - v) \chi_{\delta}) \|_\Omega  
 &\lesssim 
  \| (\nabla (v^e - v)) \chi_{\delta} \|_{U_{\delta}(\Gamma)}  
  +    \| (v^e - v) (\nabla \chi_{\delta} )\|_{U_{\delta}(\Gamma)}  
\\  
 &\lesssim 
  \| \nabla (v^e - v) \|_{U_{\delta}(\Gamma)}  
  +    \delta^{-1} \| v^e - v \|_{U_{\delta}(\Gamma)}  
 \\  
 &\lesssim 
  \| (\nablas v)^e \|_{U_{\delta}(\Gamma)} 
+  
  \| \nabla v \|_{U_{\delta}(\Gamma)}  
  +  \| n^e \cdot \nabla v  \|_{U_{\delta}(\Gamma)}  
  \\  
 &\lesssim 
\delta^{1/2}  \| \nablas v \|_{\Gamma} 
+  
  \| \nabla v \|_{U_{\delta}(\Gamma)}   
 \\
 &\lesssim 
 h_\Gamma^{1/2}  ( \| v \|_{H^2(\Omega_1)} +  \| v \|_{H^2(\Omega_2)}
 ). \label{eq:halfpower}
\end{align}
Here we used that by the properties of the extension there holds
\begin{equation}\label{eq:extension-prop-a}
\delta^{-1} \| v^e - v \|_{U_{\delta}(\Gamma)}  \lesssim \| n^e \cdot
\nabla v  \|_{U_{\delta}(\Gamma)}  
\end{equation}
see the Appendix of \cite{BurHanLar15} for a verification, and 
\begin{equation}
\| w^e \|_{U_{\delta}(\Gamma)} \lesssim \delta^{\frac12} \| w \|_\Gamma
\end{equation}
which we applied with $w=\nablas v$. Furthermore, we used the bound
\begin{equation}
 \| \nabla v \|_{U_{\delta}(\Gamma)} 
  \lesssim \delta^{\frac12}
 \sup_{t \in [-\delta,\delta]} \| \nabla v \|_{\Gamma_t} 
\end{equation}
where $\Gamma_t = \rho^{-1}(t) = \{ x \in \IR^d : \rho(x) = t\}$, for $|t|<\delta_0$, 
followed by the trace inequality 
\begin{equation}
\| \nabla v \|_{\Gamma_t} 
\leq C_t \| v \|_{H^2(\Omega_i \setminus U_{|t|}(\Gamma))} 
\leq \underbrace{\sup_{t \in [-\delta,\delta]} C_t}_{\leq C} \| v \|_{H^2(\Omega_i)}  
\end{equation}
where $i=1$ for $t \in [-\delta,0)$, $i=2$ for $t \in [0,\delta]$, and 
finally $\delta \sim h_\Gamma$.

\paragraph{Term $\bfI\bfI$.} Using the interpolation error 
estimate (\ref{eq:interpol-sz}) we obtain
\begin{align}
&\|\nabla (I-\pi_h) ( v + (v^e - v)\chi_{\delta} ) \|_\Omega
\\
&\qquad \lesssim 
\|\nabla (I-\pi_h) v \|_\Omega
+
\|\nabla (I-\pi_h) ( (v^e - v)\chi_{\delta} ) \|_\Omega
\\
&\qquad \lesssim 
\|\nabla (I-\pi_h) v \|_\Omega
+
\underbrace{\|\nabla ( (v^e - v)\chi_{\delta})  \|_\Omega}_{I}
\\
&\qquad \lesssim 
(h_\Gamma^{1/2}   + h)  (\| v \|_{H^2(\Omega_1)} + \| v \|_{H^2(\Omega_2)} )
+
 h_\Gamma^{1/2}  ( \| v \|_{H^2(\Omega_1)} +  \| v \|_{H^2(\Omega_2)} )
\end{align}
Here we used the estimate 
\begin{align}
\| \nabla (I-\pi_h) v \|_\Omega 
&\lesssim 
 \| \nabla (I-\pi_h) v \|_{\mcN_h(\mcT_h(\Gamma))} 
 +  
 \| \nabla (I-\pi_h) v \|_{\mcT_h \setminus \mcN_h(\mcT_h(\Gamma))} 
\\
&\lesssim  \| \nabla v \|_{\mcN_h(\mcT_h(\Gamma))} 
 +  
 h \| \nabla^2 v \|_{\mcT_h \setminus \mcN_h(\mcT_h(\Gamma))} 
\\
&\lesssim \delta^{1/2} \sup_{t \in [-\delta,\delta]} \| \nabla v \|_{\Gamma_t} 
 +  
 h \| \nabla^2 v \|_{\mcT_h \setminus \mcN_h(\mcT_h(\Gamma))} 
\\
&\lesssim h_\Gamma^{1/2} 
( \|  v \|_{H^2(\Omega_1)} + \|  v \|_{H^2(\Omega_2)} ) 
\end{align}
that is obtained using similar arguments as in \eqref{eq:halfpower}.
\paragraph{Term $\bfI\bfI\bfI$.} Using the trace inequality 
\begin{equation}
\| w \|^2_{\Gamma \cap T} \lesssim h^{-1} \| v \|^2_T + h \|\nabla v \|^2_T 
\end{equation}
see \cite{HanHanLar},  the interpolation estimate (\ref{eq:interpol-sz}), 
the fact (\ref{eq:chideltaequalsone}), and finally the stability of the 
extension we find that 
\begin{align}\nonumber
&\|\nablas (I-\pi_h) ( v + (v^e - v)\chi_{\delta} ) \|^2_\Gamma 
\\
&\qquad =
 \|\nablas ( (I-\pi_h) v^e) \|^2_\Gamma
\\ 
&\qquad \lesssim 
h^{-1}_\Gamma\| \nabla ( (I-\pi_h)  ( v + (v^e - v)\chi_{\delta} )  ) \|^2_{\mcT_h(\Gamma)} 
\\ \nonumber
&\qquad \qquad + 
h_\Gamma \| \nabla^2 (  (I-\pi_h)  ( v + (v^e - v)\chi_{\delta} ) ) \|^2_{\mcT_h(\Gamma)}
\\
&\qquad \lesssim 
h_\Gamma \| \nabla^2  ( v + (v^e - v)\chi_{\delta} ) \|^2_{\mcN_h(\mcT_h(\Gamma))}
\\
&\qquad \lesssim 
h_\Gamma \| v^e  \|^2_{H^2(\mcN_h(\mcT_h(\Gamma)))}
\\
&\qquad \lesssim h^2_\Gamma \| v \|^2_{H^2(\Gamma)}
\end{align}
\end{proof}

\begin{rem} Alternatively we may use a different extension operator and prove an interpolation 
estimate which requires less regularity as follows. We include some details for convenience
\begin{itemize}
\item There is a continuous extension operator 
\begin{equation}
H^{s}(\Gamma) \ni v \mapsto v^E \in H^{s+1/2}(\Omega)
\end{equation}
We construct $v^E$ by first solving the Dirichlet problem $\Delta v^E = 0$ in $\Omega_2$ and 
$v^E = v$ on $\Gamma$, for which we have the regularity estimate 
$\| v^E \|_{H^{s+1/2}(\Omega_2)} \lesssim \| v\|_{H^s(\Gamma)}$. Next we extend $v^E$ 
to $\IR^d$ using a standard continuous extension operator $\mathcal{E}_{\Omega_2}:H^s(\Omega_2) \rightarrow H^s(\IR^d)$, $s>0$, that is $v^E |_{\IR^d \setminus \Omega_2}  = \mathcal{E}_{\Omega_2} 
(v^E |_{\Omega_2}$). 

\item With $v^E$ instead of $v^e$ in the definition of $I_h$ we derive the interpolation 
estimate  
\begin{equation}
\tn v - I_h v \tn
\lesssim  
(h + h_\Gamma^{1/2} )(\| v \|_{H^2(\Omega)} 
+ \| v \|_{H^2(\Gamma)} )
\end{equation}
as follows. Term $I$ and $II$ can be estimated in the same way as above. For Term $III$ 
we have the estimates
\begin{align}\nonumber
&\|\nablas (I-\pi_h) ( v + (v^E - v)\chi_{\delta} ) \|^2_\Gamma 
\\
&\qquad =
 \|\nablas ( (I-\pi_h) v^E ) \|^2_\Gamma
\\ 
&\qquad \lesssim 
h^{-1}_\Gamma\| \nabla ( (I-\pi_h)  ( v + (v^E - v)\chi_{\delta} )  ) \|^2_{\mcT_h(\Gamma)} 
\\ \nonumber
&\qquad \qquad + 
h_\Gamma \| \nabla^2 (  (I-\pi_h)  ( v + (v^E - v)\chi_{\delta} ) ) \|^2_{\mcT_h(\Gamma)}
\\
&\qquad \lesssim 
h_\Gamma \| \nabla^2  ( v + (v^E - v)\chi_{\delta} ) \|^2_{\mcN_h(\mcT_h(\Gamma))}
\\
&\qquad \lesssim 
h_\Gamma \| v^E  \|^2_{H^2(\mcN_h(\mcT_h(\Gamma)))}
\\
&\qquad \lesssim h_\Gamma \| v \|^2_{H^{3/2}(\Gamma)} 
\\
&\qquad \lesssim h_\Gamma \| v \|^2_{H^2(\Omega_2)}
\end{align}
where at last we used a trace inequality to pass from $\Gamma$ to  $\Omega_2$. 
\end{itemize}
\end{rem}

\subsection{Error Estimates}

\begin{thm} The following error estimates hold
\begin{align}\label{eq:errorest-energy}
\tn u - u_h \tn \lesssim (h_\Gamma^{1/2} + h) 
\Big(\| u \|_{H^2(\Omega_1)}  + \| u \|_{H^2(\Omega_2)} \Big) 
+ 
h_\Gamma \| u \|_{H^2(\Gamma)}  
\end{align}
\begin{equation}\label{eq:errorest-L2}
\| u - u_h \|_\Omega + \| u - u_h \|_\Gamma 
\lesssim 
 (h_\Gamma + h^2) 
\Big(\| u \|_{H^2(\Omega_1)}  + \| u \|_{H^2(\Omega_2)} \Big) 
+ 
h_\Gamma^2 \| u \|_{H^2(\Gamma)}  
\end{equation}
\end{thm}
\begin{proof} {\bf (\ref{eq:errorest-energy}).} The proof follows immediately from 
Galerkin orthogonality and the interpolation error estimate 
\begin{align}
\tn u - u_h \tn^2 &=A(u - u_h, u-u_h) 
\\
&= A(u - u_h, u - \pi_h u ) 
\\
&\leq \tn u - u_h \tn \, \tn u - \pi_h u \tn
\end{align}
and thus 
\begin{align}
\tn u - u_h \tn 
&\leq \tn u - \pi_h u \tn
\\
&\lesssim 
 (h_\Gamma^{1/2} + h) 
\Big(\| u \|_{H^2(\Omega_1)}  + \| u \|_{H^2(\Omega_2)} \Big) 
+ 
h_\Gamma \| u \|_{H^2(\Gamma)}  
\end{align}

\paragraph{(\ref{eq:errorest-L2}).}
For the $L^2$  estimate we obtain an error representation formula 
using the dual problem: find $\phi \in V$ such that 
\begin{equation}
A(v,\phi) = (u-u_h,v)_{\Omega} +  (u-u_h,v)_{\Gamma}\qquad \forall v \in V
\end{equation}
with $v = u - h_h$,  
\begin{align}
\nonumber
&\| u - u_h \|^2_\Omega + \| u - u_h \|^2_\Gamma 
\\
&\qquad =
A(u - u_h, \phi ) 
\\
&\qquad = 
A(u-u_h,\phi - I_h \phi )
\\
&\qquad \leq 
\tn u - u_h \tn \, \tn \phi - I_h \phi \tn
\\
&\qquad 
\lesssim
\Big( (h_\Gamma^{1/2} + h) 
\Big(\| u \|_{H^2(\Omega_1)}  + \| u \|_{H^2(\Omega_2)} \Big) 
+ 
h_\Gamma \| u \|_{H^2(\Gamma)}  \Big)
\\ \nonumber
&\qquad \qquad \times
\Big( (h_\Gamma^{1/2} + h) 
\Big(\| \phi \|_{H^2(\Omega_1)}  + \| \phi \|_{H^2(\Omega_2)} \Big) 
+ 
h_\Gamma \| \phi \|_{H^2(\Gamma)}  \Big)
\\
&\qquad
\lesssim 
\Big( (h_\Gamma+ h^2) 
\Big(\| u \|_{H^2(\Omega_1)}  + \| u \|_{H^2(\Omega_2)} \Big) 
+ 
h_\Gamma^2 \| u \|_{H^2(\Gamma)}  \Big)
\\ \nonumber 
&\qquad \qquad 
\times \Big( \| u - u_h \|^2_\Omega + \| u - u_h \|^2_\Gamma \Big)^{1/2}
\end{align}
where at last we used the elliptic regularity estimate (\ref{eq:ell-reg}).
\end{proof}

\section{Numerical Examples}
\subsection{A Convergence Study for a Simple Interface Problem}

We consider a problem with $f=0$, $f_\Gamma= 1$, $a_1=a_2=a_\Gamma=1$ on the domain $\Omega=(1,e^{5/4})\times(1,e^{5/4})$
with a crack at $\sqrt{(x^2+y^2)}=: r = e$. The exact solution to this problem is given as
\[\begin{array}{>{\displaystyle}l}
 u_1 = \frac{\log{(r)}}{5}\left(4+e\right)\quad \text{for}\quad 1< r < e, \\[4mm]
 u_2 = \frac{4-4 e}{5}\left(\log{(r)}-\frac{5}{4}\right)+1\quad \text{for}\quad e <r < e^{5/4},
 \end{array}\]
and this solution is applied as Dirichlet boundary conditions on $\partial\Omega$, corresponding to a
solution depending only on $r$ with $u=u_1=0$ at $r=1$ and $u=u_2=1$ at $r=e^{5/4}$. We compare the convergence on a globally refined mesh with a mesh 
which is locally refined so that $h_\Gamma \leq h^2$ at $\Gamma$. The convergence is then checked in $L^2$ norm and $H^1$ (semi-) norm. In Figure \ref{elev}
we show the discrete solution on a given locally refined mesh. We note that optimal convergence is obtained at the cost of locally refining the mesh, Figure \ref{conv2}, whereas
a globally refined mesh gives suboptimal convergence in accordance with (\ref{eq:errorest-energy}) and (\ref{eq:errorest-L2}), Figure \ref{conv1}.

\subsection{A More Complex Example with a Bifurcating Crack}

In this example we illustrate the modeling capabilities of our approach with application 
to a more complex problem involving a bifurcating crack. 

\paragraph{Model Problem.}
Let us for simplicity consider 
a two dimensional problem with a one dimensional crack $\Gamma$ which can be described as 
a graph with nodes $\mathcal{N} = \{ x_i \}_{i\in I_N}$ and edges 
$\mathcal{G} =\{\Gamma_j\}_{j\in I_G}$, where $I_N$, $I_G$ are finite index sets, and each 
$\Gamma_j$ is a curve between two nodes with indexes $I_N(j)$. For each $i\in I_N$ we 
let $I_G(i)$ be the set of indexes corresponding to curves for which $x_i$ is an end point. See 
Figures \ref{fig:bifurcating-layers-notation} and \ref{fig:nodes-curves-tangents}.

The governing equations are given by (\ref{eq:strong-problem-a})--(\ref{eq:strong-problem-d}) 
together with two conditions at each of the nodes $x_i \in \mathcal{N}$, the continuity condition 
\begin{equation}
u_{\Gamma_k} (x_i) = u_{\Gamma_l}(x_i) \qquad \forall k,l \in I_G(i) 
\end{equation}
and the Kirchhoff condition 
\begin{equation}\label{eq:kirchhoff}
\sum_{j \in I_{\mathcal{G}}(i)} (t_{\Gamma_j} \cdot a_{\Gamma_j} \nabla_{\Gamma_j} u_{\Gamma_j}) |_{x_j} =0
\end{equation}
where $t_{\Gamma_j} (x_i)$ is the exterior tangent unit vector to $\Gamma_j$ at $x_i$. 

\paragraph{Finite Element Method.} Let 
$V_\Gamma = \{v \in C(\Gamma) : v \in H^1(\Gamma_j), j \in I_G\}$ 
and  $V = H^1_0(\Omega) \cap V_\Gamma$. We proceed as in the derivation of the weak form 
in the standard case  (\ref{eq:weak-start})--(\ref{eq:weak-end}). However, when we use Green's 
formula on $\Gamma$ we proceed segment by segment as follows
\begin{align}\nonumber
&\sum_{j \in I_G} -(\nabla_{\Gamma_j} \cdot a_{\Gamma_j} \nabla_{\Gamma_j} u,v)_{\Gamma_j} 
\\
&\qquad=
\sum_{j \in I_G} ( a_{\Gamma_j} \nabla_{\Gamma_j} u,\nabla_{\Gamma_j} v)_{\Gamma_j} 
- \sum_{j \in I_G} \sum_{i \in I_N(j)} ( t_i \cdot a_{\Gamma_j}\nabla_{\Gamma_j} u,v)_{x_i}
\\
&\qquad=
\sum_{j \in I_G} ( a_{\Gamma_j} \nabla_{\Gamma_j} u,\nabla_{\Gamma_j} v)_{\Gamma_j} 
\end{align} 
where we changed the order of summation and used the Kirchhoff condition (\ref{eq:kirchhoff})  
together with the fact $v$ is continuous to conclude that 
\begin{align}
\sum_{j \in I_G} \sum_{i \in I_N(j)} ( t_i \cdot a_{\Gamma_j} \nabla_{\Gamma_j} u,v)_{x_i}
&=
\sum_{i \in I_N(j)}  \sum_{j \in I_G}( t_i \cdot a_{\Gamma_j}\nabla_{\Gamma_j} u,v)_{x_i} = 0
\\
&=
\sum_{i \in I_N(j)}   \underbrace{\Big( \sum_{j \in I_G} (t_i \cdot a_{\Gamma_j}\nabla_{\Gamma_j} u)|_{x_i} \Big)}_{=0} v(x_i)  
\end{align}
Thus we conclude that: 
\begin{itemize}
\item The weak formulation is precisely the same in the bifurcating crack case 
as in the standard case (\ref{eq:weak}). 
\item Since $V_h \subset V$ the method also takes the same form as 
in the standard case (\ref{eq:fem}) in this more complex situation.
\end{itemize}
The similar derivation can be performed for a two dimensional bifurcating 
crack embedded into $\IR^3$, see \cite{HanJonLarLar17} for further details.

\paragraph{Numerical Example.} The crack pattern is modeled using a polygonal chain 
interpolating higher order curves with each part of the chain of length $h/10$. The intersection 
points with element sides are computed and a new polygonal chain containing the old one 
cut by the intersection points is constructed. In Figure \ref{meshes} we show the effect on a 
coarse mesh and on a locally refined mesh. We now compute two different solutions using global refinement and local refinement. We use local refinement at $\Gamma$ until the smallest meshsize 
equals that of the globally refined model. In Figure \ref{elevations} we give the  computed solutions using these two approaches. Here $a_1=a_2=1$ and $a_\Gamma = 100$, $f=f_\Gamma=0$, and we impose, on the domain 
$\Omega =(0,13)\times(0,9.5)$,
$u=1$ at $x=0$ and $u=0$ at $x=13$ and homogeneous Neumann boundary conditions at $y=0$ and $y=9.5$. The corresponding solution with $\alpha_\Gamma=0$ is thus a plane.

\section{Concluding Remarks} 

We suggest a continuous finite element method with superimposed lower-dimensional
features modeling interfaces. The effect of these are computed using the higher dimensional 
basis functions and added to the stiffness matrix so as to yield further ``stiffness'' to the problem. 
Due to the fact that we cannot resolve kinks in the normal derivative across the interface
we do not obtain optimal convergence orders. We propose a simple adaptive scheme based 
on an a priori error estimate which guides the choice of optimal local
mesh size, to improve the local accuracy, regaining the optimal order of convergence. The resulting scheme is very simple 
and computationally expedient for many applications such as when optimization of the position 
of interfaces is of interest.

\newpage


\begin{figure}
\begin{center}
\includegraphics[scale=0.4]{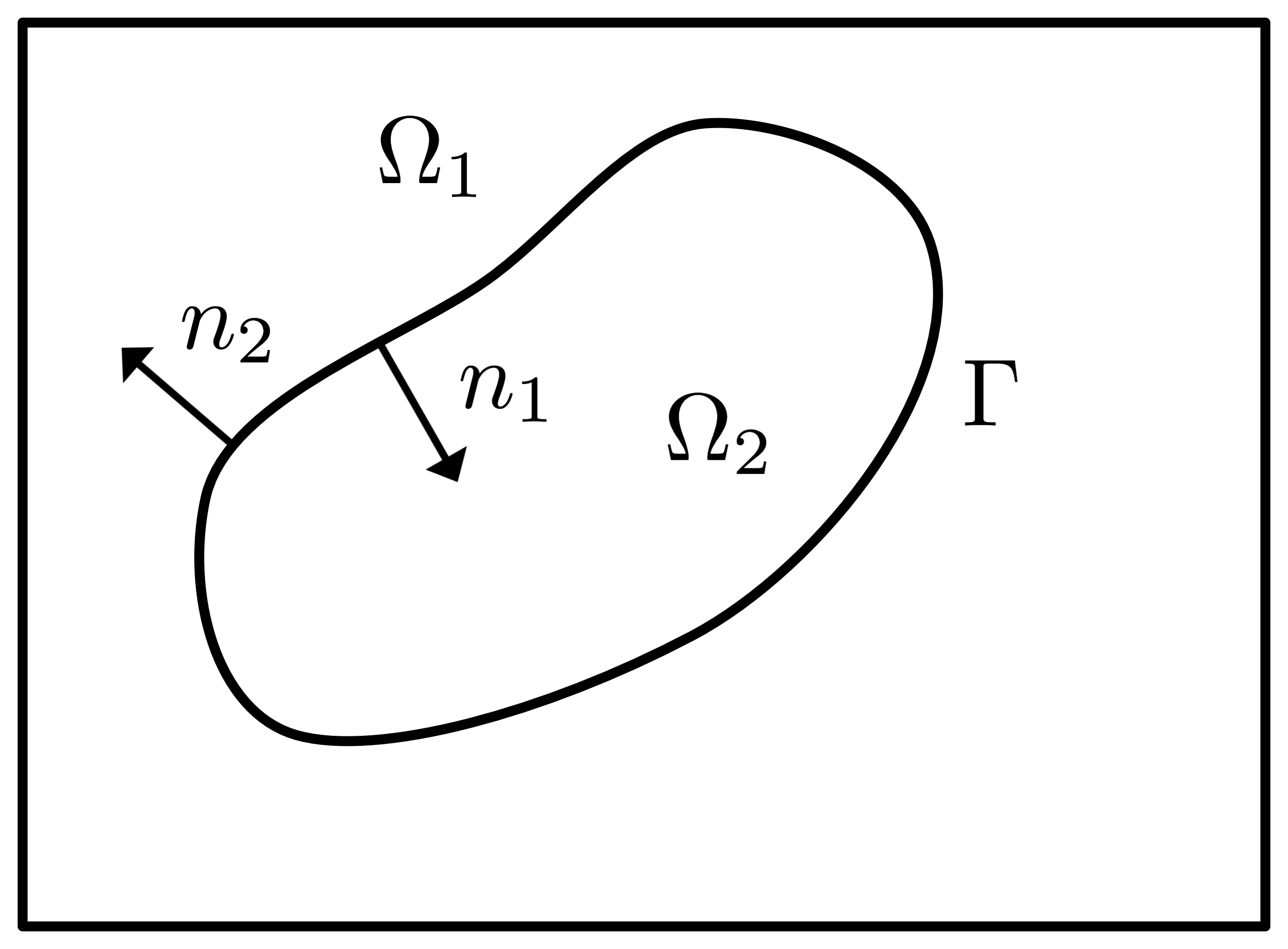}
\end{center}
\caption{\label{fig:domains}
The domains $\Omega_1$, $\Omega_2$, the interface $\Gamma$, and the unit 
exterior normals $n_1$ and $n_2$.}
\end{figure}


\begin{figure}
\begin{center}
\includegraphics[scale=0.2]{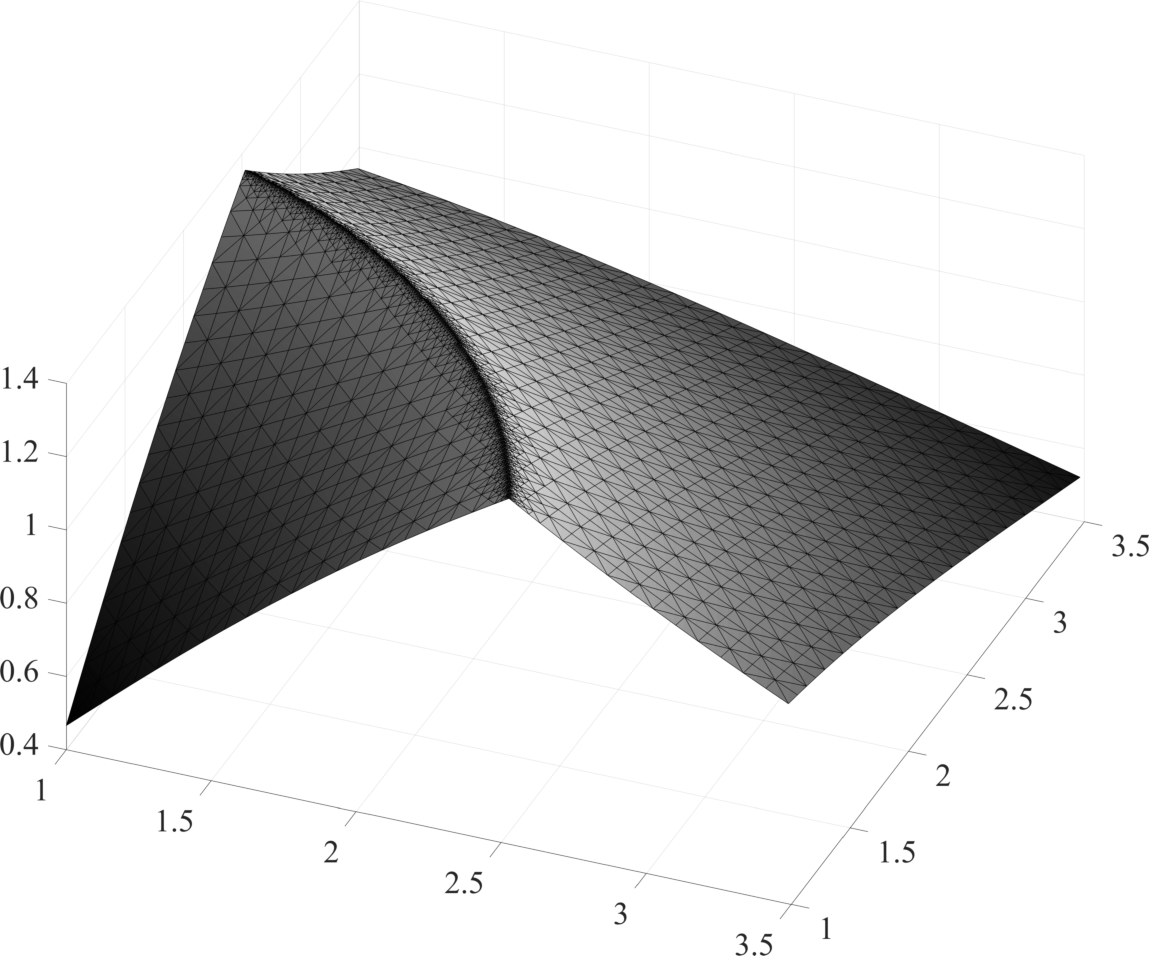}
\end{center}
\caption{Elevation of the solution on a locally refined mesh.\label{elev}}
\end{figure}

\begin{figure}
\begin{center}
\includegraphics[scale=0.3]{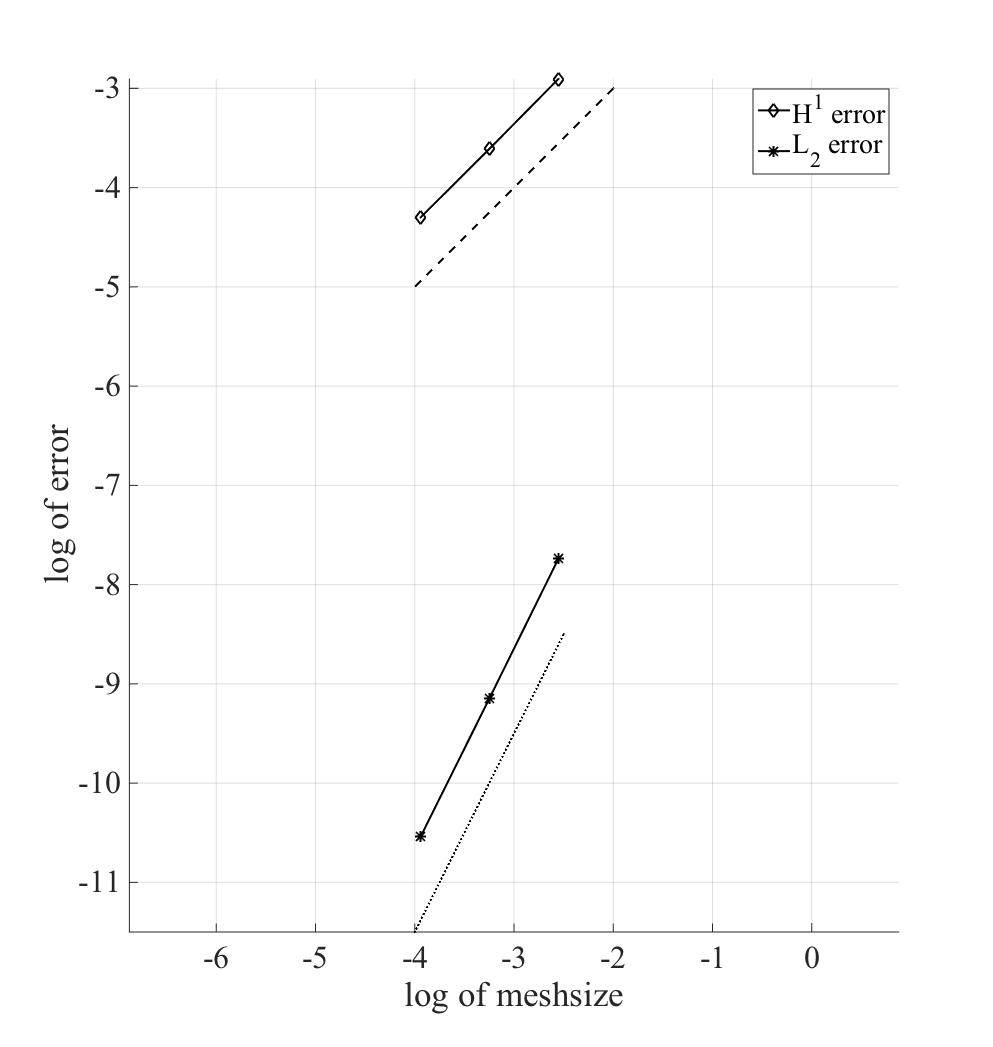}
\end{center}
\caption{Convergence on a locally refined mesh. Dashed line has inclination 1:1, and dotted line 2:1.\label{conv2}}
\end{figure}
\begin{figure}
\begin{center}
\includegraphics[scale=0.3]{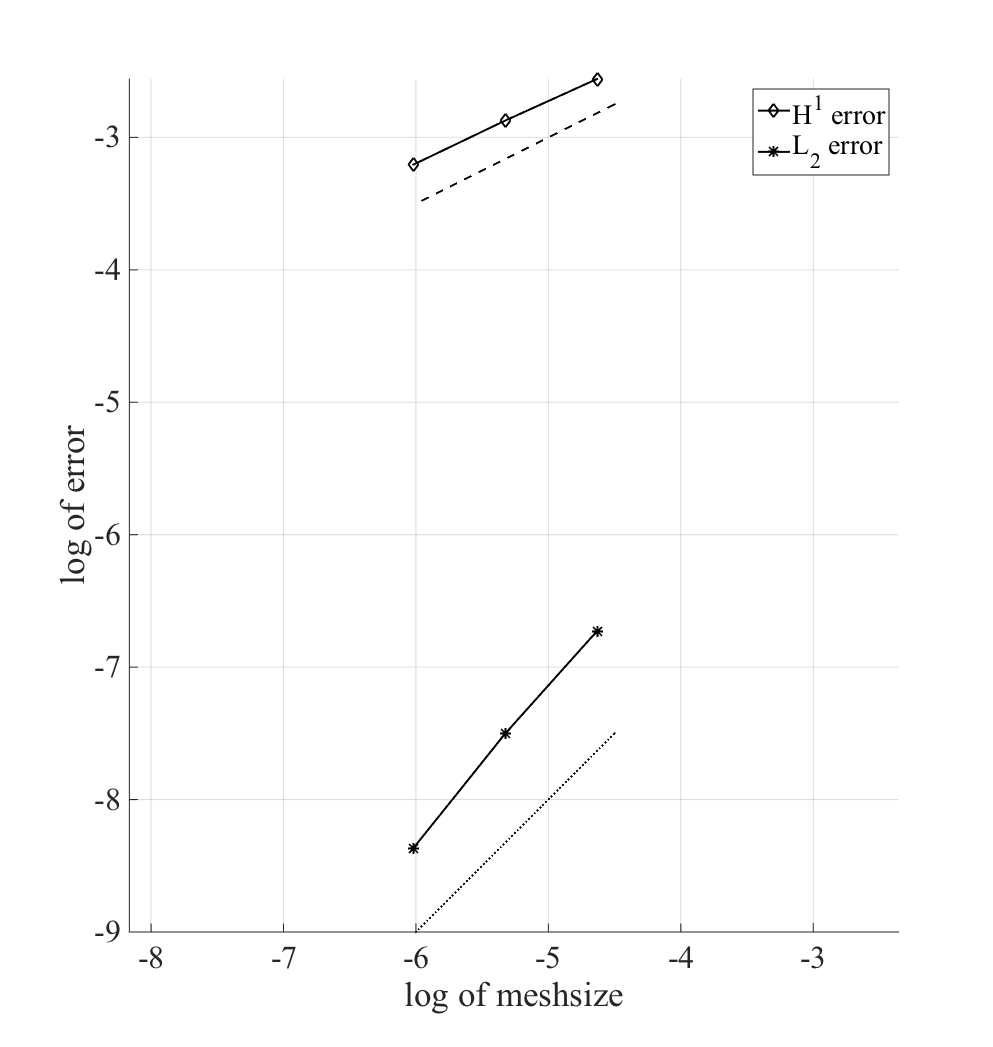}
\end{center}
\caption{Convergence on a globally refined mesh. Dashed line has inclination 1:2, and dotted line 1:1.\label{conv1}}
\end{figure}



\begin{figure}
\begin{center}
\includegraphics[scale=0.4]{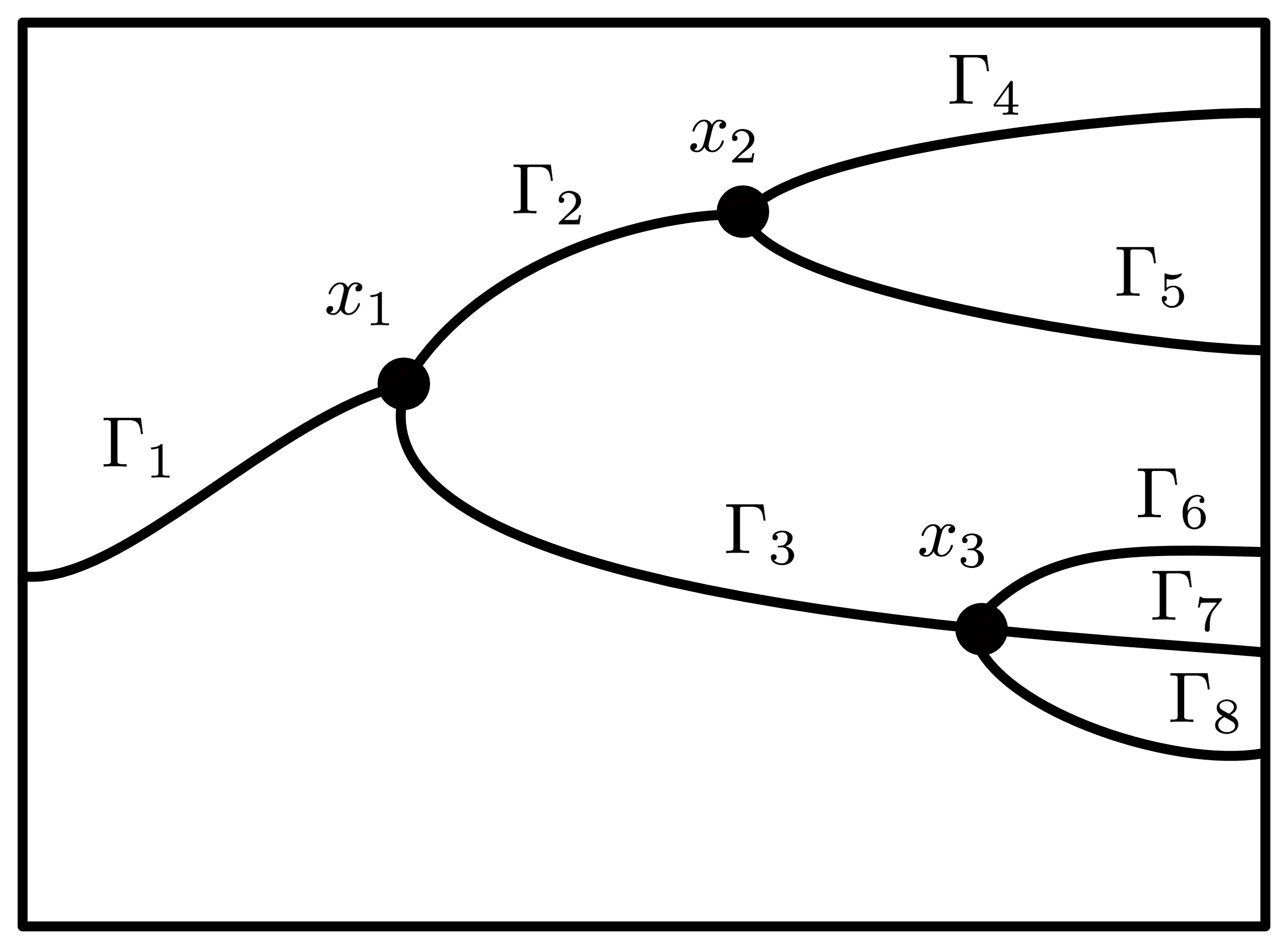}
\end{center}
\caption{\label{fig:bifurcating-layers-notation}
Schematic figure of bifurcating cracks with nodes 
$\mathcal{N}=\{x_i\}_{i=1}^3$ and curves $\mathcal{G}=\{\Gamma_i\}_{i=1}^8$. The connectivity 
is described by the mappings $I_N$ and $I_G$ and we have for instance $I_N(3) = \{1,3\}$ and 
$I_G(2)=\{2,4,5\}$.}
\end{figure}

\begin{figure}
\begin{center}
\includegraphics[scale=0.4]{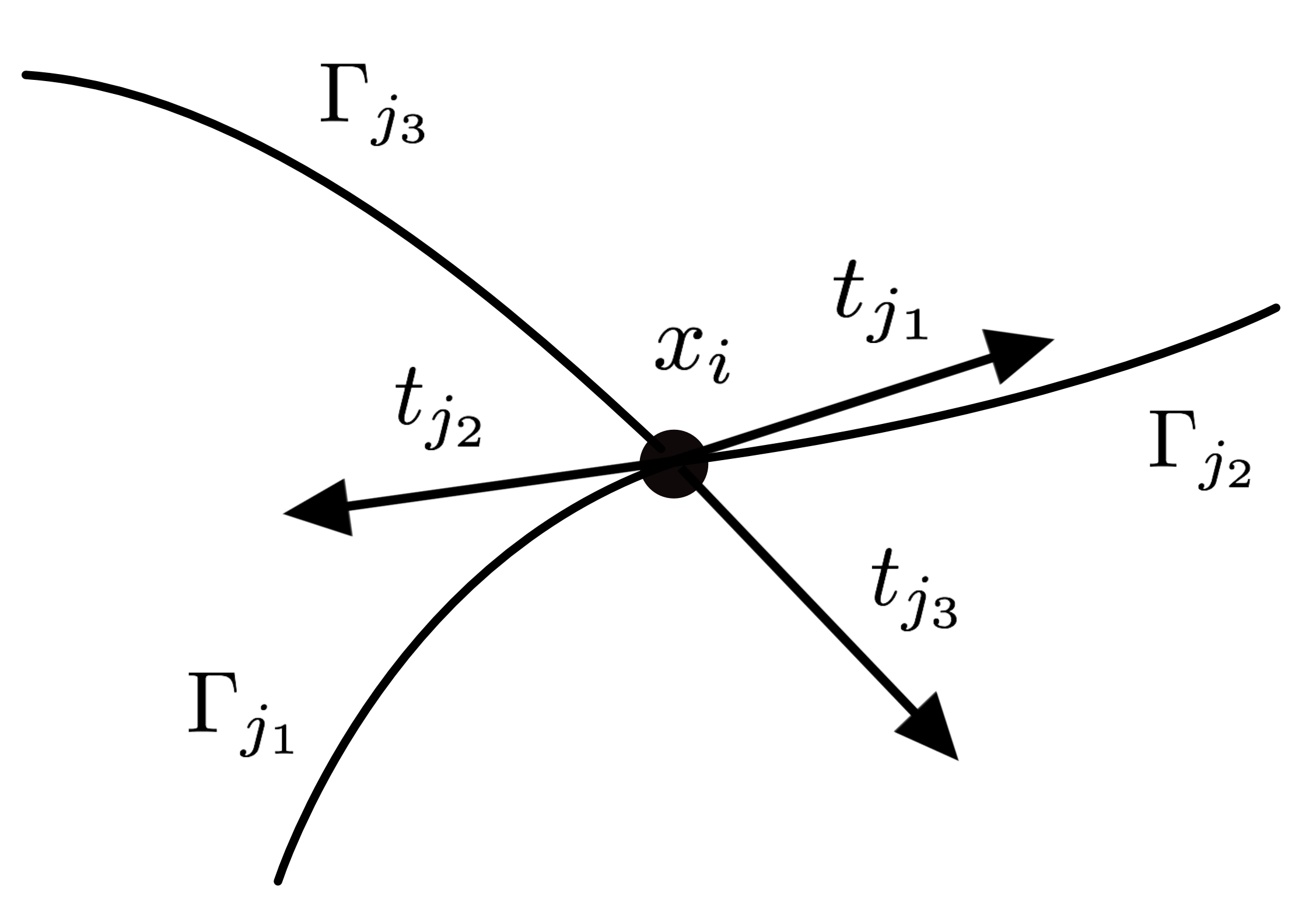}
\end{center}
\caption{\label{fig:nodes-curves-tangents}
Schematic figure of node $x_i$ with its associated three curves 
$\Gamma_k$, and exterior unit tangents $t_k$ at $x_i$ for $k \in I_G(i) = \{j_1,j_2,j_3\}$.}
\end{figure}

\begin{figure}
\begin{center}
\includegraphics[scale=0.2]{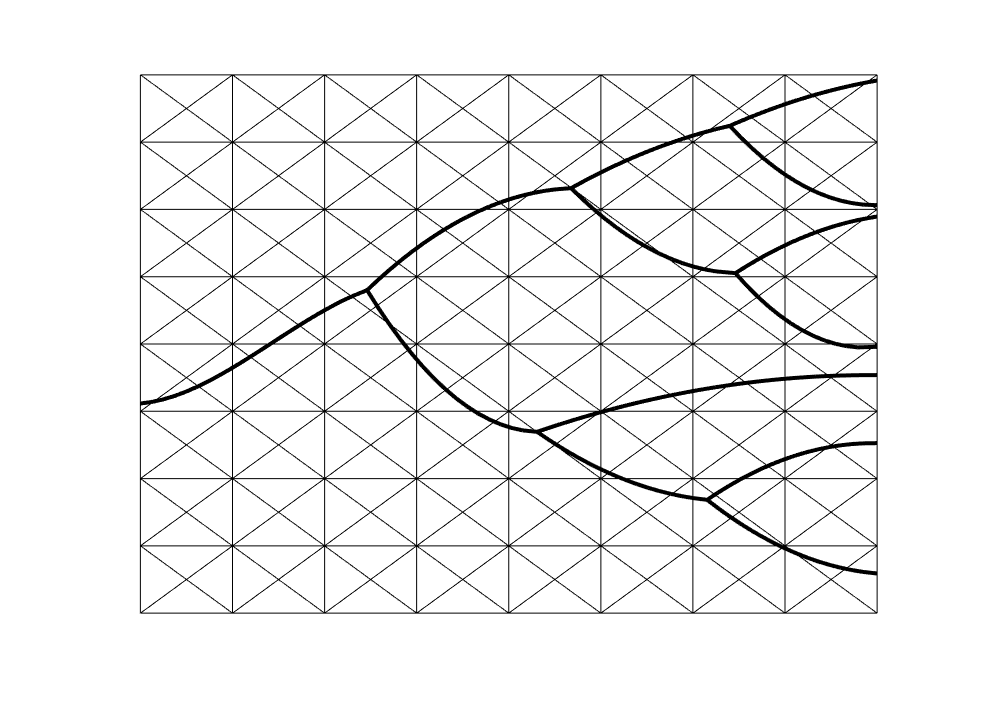}
\includegraphics[scale=0.15]{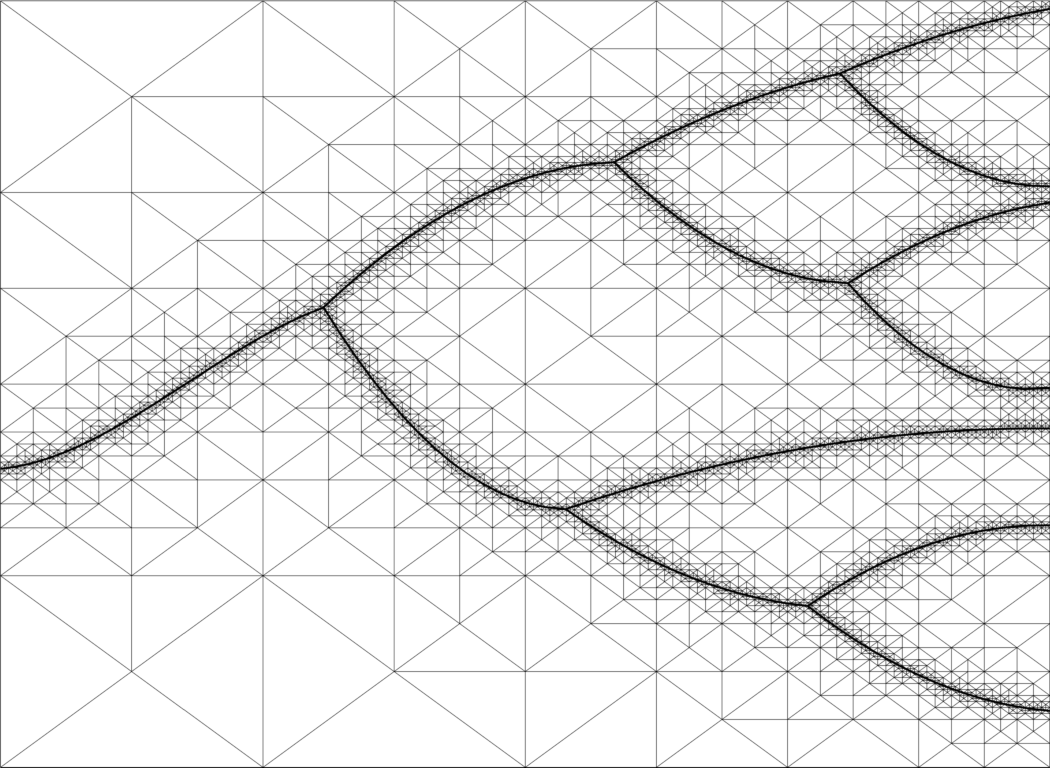}
\end{center}
\caption{Crack pattern modelled on a coarse and a locally refined mesh.\label{meshes}}
\end{figure}

\begin{figure}
\begin{center}
\includegraphics[scale=0.22]{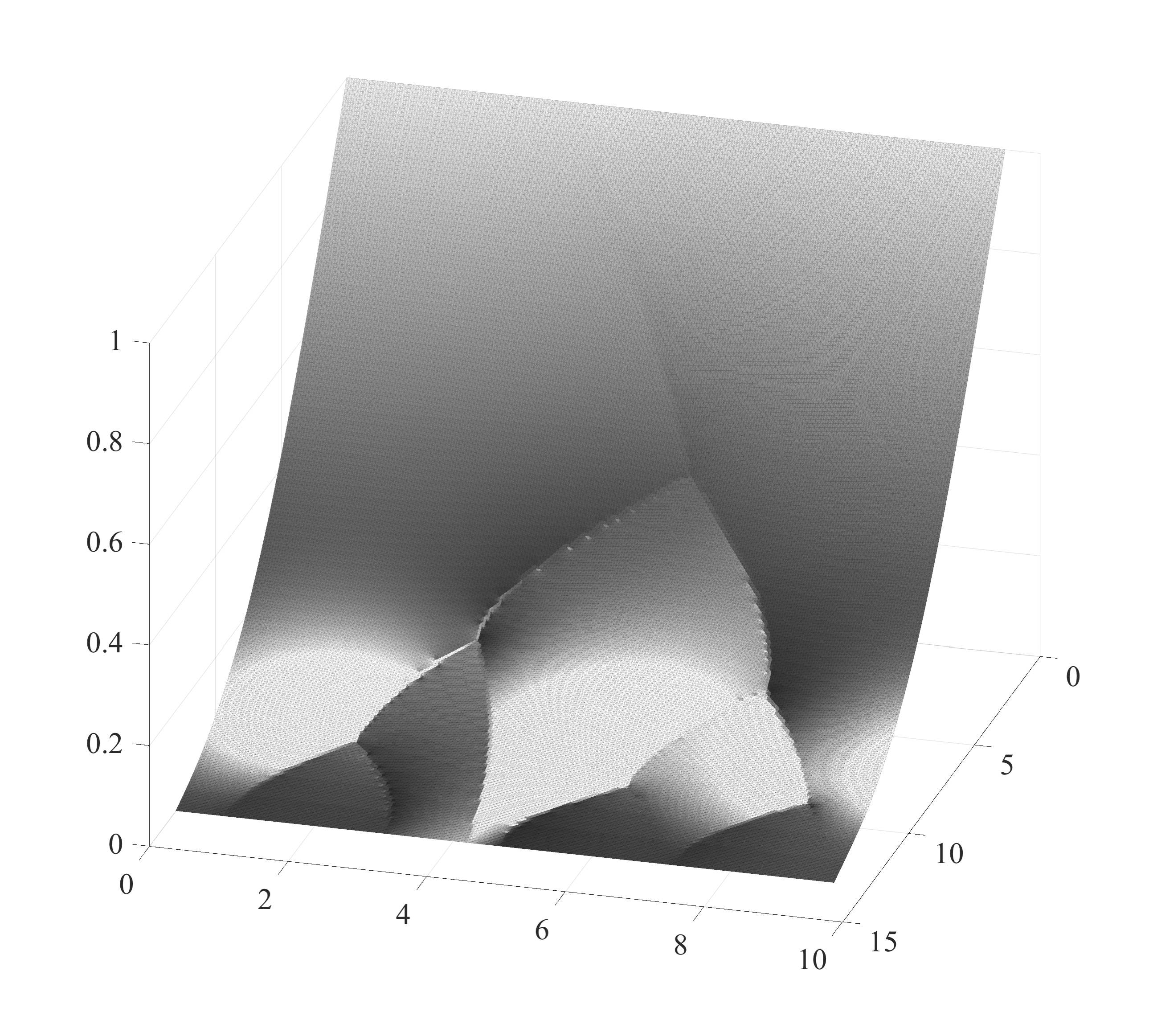}
\includegraphics[scale=0.2]{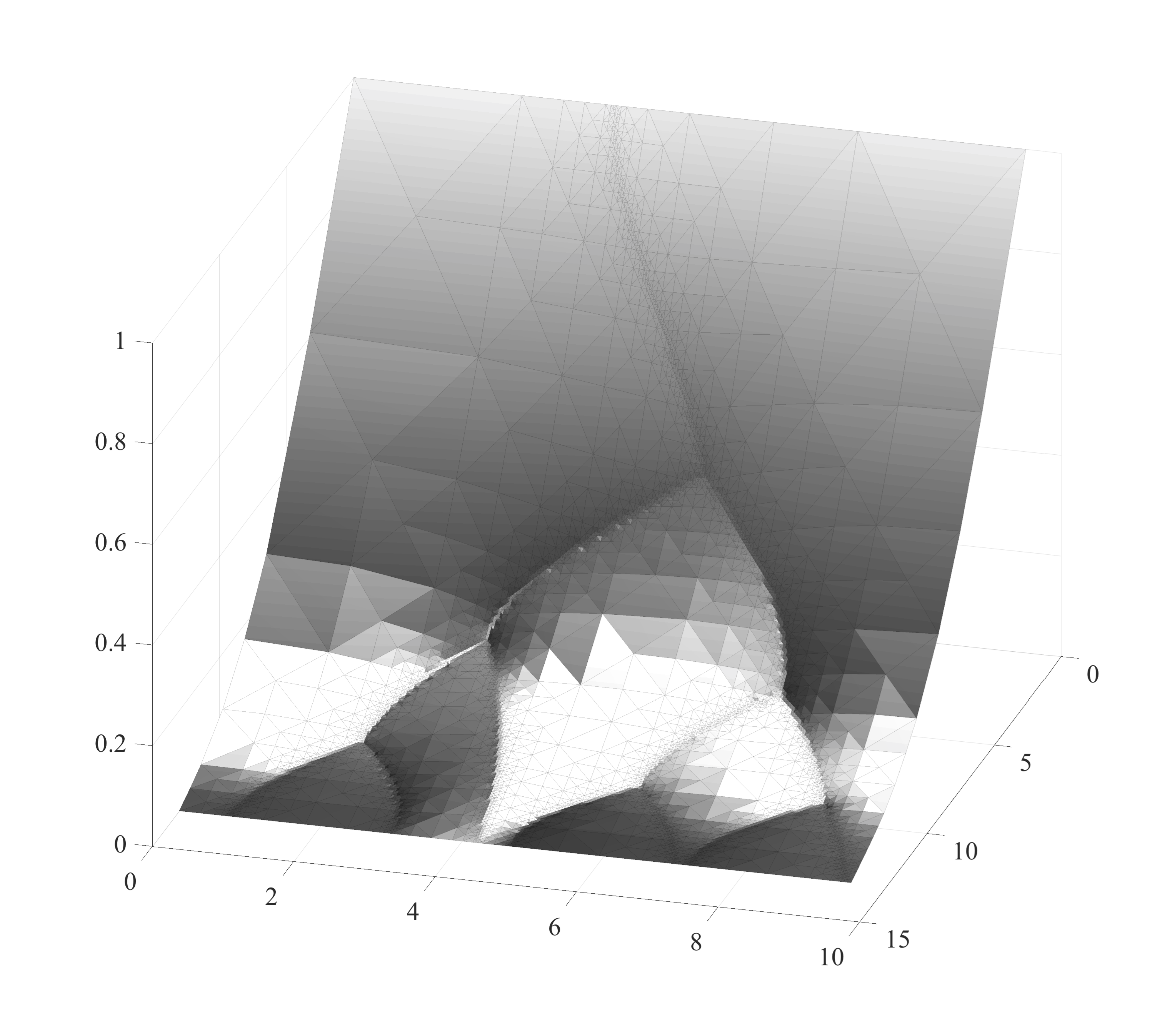}
\end{center}
\caption{Discrete solutions on a coarse and a locally refined mesh.\label{elevations}}
\end{figure}

\end{document}